\documentclass[12pt,reqno]{article}
\usepackage{amssymb,amscd,amsmath,amsthm,color}
\newtheorem{theorem}{Theorem}[section]
\newtheorem{lemma}[theorem]{Lemma}
\newtheorem{cor}[theorem]{Corollary}
\newtheorem{prop}[theorem]{Proposition}

\usepackage{enumerate,amssymb}
\theoremstyle{definition}

\theoremstyle{remark}
\newtheorem{remark}[theorem]{Remark}

\numberwithin{equation}{section}

\def\bM{\mathbb{M}}

\begin{document}
\baselineskip=15pt

\title{Matrix inequalities and majorizations around Hermite-Hadamard's inequality}

\author{ Jean-Christophe Bourin{\footnote{Funded by the ANR Projet (No.\ ANR-19-CE40-0002) and by the French Investissements
 d'Avenir program, project ISITE-BFC (contract ANR-15-IDEX-03).}} \,and Eun-Young Lee{\footnote{This research was supported by
Basic Science Research Program through the National Research
Foundation of Korea (NRF) funded by the Ministry of
Education (NRF-2018R1D1A3B07043682)}  }   }

\date{ }

\maketitle

\vskip 10pt\noindent
{\small 
{\bf Abstract.} We study the classical Hermite-Hadamard inequality in the matrix setting.
This leads to a number of interesting matrix inequalities such as the Schatten $p$-norm estimates
$$
\left(\|A^q\|_p^p + \|B^q\|_p^p\right)^{1/p} \le \|(xA+(1-x)B))^q\|_p+ \|(1-x)A+xB)^q\|_p
$$
for all  positive (semidefinite) $n\times n$ matrices $A,B$ and $0<q,x<1$. A related   decomposition, with the assumption $X^*X+Y^*Y=XX^*+YY^*=I$, is
$$
(X^*AX+Y^*BY)\oplus (Y^*AY+X^*BX) =\frac{1}{2n}\sum_{k=1}^{2n} U_k (A\oplus B)U_k^*
$$
for some family   of $2n\times 2n$ unitary matrices $U_k$.  This is a majorization which is obtained by using the Hansen-Pedersen trace inequality.

\vskip 5pt\noindent
{\it Keywords.}     Positive definite matrices, block matrices, convex functions, matrix inequalities. 
\vskip 5pt\noindent
{\it 2010 mathematics subject classification.} 15A18, 15A60,  47A30.

}

\section{Elementary scalar inequalities }

Extending basic scalar inequalities, for instance $|a+b|\le |a|+|b|$, to matrices lies at the very heart of matrix analysis. Here, we are interested in the elementary inequality which supports the Hermite-Hadamard inequality. 
This classical  theorem can be stated  as follows:

\vskip 5pt
\begin{prop}\label{propHH1}  Let $f(t)$ be a convex function defined on the interval $[a,b]$. Then,
$$
f\left(\frac{a+b}{2}\right)\ \le \int_0^1 f((1-x)a+xb)\, {\mathrm{d}}x\le \frac{f(a)+f(b)}{2}.
$$
\end{prop}

In spite of its simplicity, the Hermite-Hadamard inequality is a powerful tool for deriving a number of important inequalities; see the nice paper \cite{Ni} and references therein.

The first inequality immediately follows from the convexity assumption
\begin{equation}\label{eqbasic}
f\left(\frac{a+b}{2}\right) \le \frac{f((1-x)a+xb) +f(xa +(1-x)b)}{2} 
\end{equation}
The second inequality is slightly more subtle; it follows from the sums property
\begin{equation}\label{eqext1}
 f((1-x)a+xb) +f(xa +(1-x)b) \le f(a) + f(b)
\end{equation}
which requires   the convexity assumption twice.
This  is the key for Proposition \ref{propHH1} and it has a clear geometric interpretation;  \eqref{eqext1} is equivalent to the increasingness of
$$
\varphi(t) := f(m+t) +f(m-t)
$$
with $m=(a+b)/2$ and $t\in[0,b-m]$. In fact,  if we  assume that $f(t)$ is $C^2$ and observe that for $t\in[0,b-m]$,
$$
\varphi'(t)=f'(m+t)-f'(m-t) =\int_{m-t}^{m+t} f''(s)  \,{\mathrm{d}}s,
$$
we can estimate  $\varphi'(t)$ with $f''(s)\ge 0$.

The extremal property \eqref{eqext1} of  $f(t)$  says that for four points in $[a,b]$,
\begin{equation}\label{eqext2}
p\le s\le t\le q, \ p+q=s+t \Rightarrow f(s)+f(t) \le f(p)+f(q).
\end{equation}

Let us see now what can be said for matrices.
Important matrix versions of  \eqref{eqbasic} are well-known. Let $\bM_n$ denote the space of $n\times n$ matrices and   $\bM_n^{s.a}$ its self-adjoint (Hermitian) part with the usual order $\le $ induced by the positive semidefinite cone $\bM_n^+$. We recall \cite[Corollary 2.2]{BL1}.

\vskip 5pt
\begin{theorem} \label{th-convex} Let  $A,B\in \bM_n^{s.a}$ with spectra in $[a,b]$ and let   $f(t)$ be a convex function on $[a,b]$.  Then, for some unitaries $U,\,V\in\bM_n$,
\begin{equation*}
f\left(\frac{A+B}{2}\right)\le \frac{1}{2}\left\{U\frac{f(A)+f(B)}{2}U^*+V\frac{f(A)+f(B)}{2}V^*\right\}.
\end{equation*}
If furthermore $f(t)$ is monotone, then we can take $U=V$.
\end{theorem}

\vskip 5pt
Theorem \ref{th-convex} is a major improvement of the classical trace inequality of von Neumann (around 1920),
$$
{\mathrm{Tr\,}}f\left(\frac{A+B}{2}\right) \le {\mathrm{Tr\,}}\frac{f(A)+f(B)}{2}
$$
which entails the following trivial extension of Proposition \ref{propHH1}.

\vskip 5pt
\begin{prop}\label{propHH2}  Let $f(t)$ be a convex function defined on the interval $[a,b]$ and let $A,B\in\bM_n^{s.a}$ with spectra in $[a,b]$. Then,
$$
 {\mathrm{Tr\,}}f\left(\frac{A+B}{2}\right)\ \le  {\mathrm{Tr\,}}\int_0^1 f((1-x)A+xB)\, {\mathrm{d}}x\le  {\mathrm{Tr\,}}\frac{f(A)+f(B)}{2}.
$$
\end{prop}

What about matrix versions of the equivalent scalar inequalities \eqref{eqext1} and \eqref{eqext2} ? There is no hope for \eqref{eqext2} : in general the trace inequality
$$
{\mathrm{Tr}\,} f(P) + {\mathrm{Tr}\,} f(Q)  \le {\mathrm{Tr}\,} f(S) + {\mathrm{Tr\,}} f(T) 
$$
does not hold for all Hermitian matrices $P,Q,S,T$ with spectra in $[a,b]$ and such that
$$
P\le S\le T \le Q, \quad P+Q=S+T.
$$

In the matrix setting \eqref{eqext1} and \eqref{eqext2} are not equivalent. This paper aims to establish two  matrix versions of the extremal inequality \eqref{eqext1}.  Doing so, we will obtain several new matrix inequalities. 

 For an operator convex functions $h(t)$ on $[a,b]$, and $A,B\in\bM_n^{s.a}$ with spectra in this interval, the matrix version of \eqref{eqext1} (as well as Proposition \ref{propHH1}) obvioulsy holds,
\begin{equation}\label{eqext3}
 h((1-x)A+xB) +h(xA +(1-x)B) \le h(A) + h(B)
\end{equation}
for any $0<x<1$. Our results hold for much more general convex/concave functions and have  applications to eigenvalue inequalities that cannot be derived from
\eqref{eqext3}, even in the case of the simplest operator convex/concave function $h(t)=t$. 

 We often use a crucial assumption: our functions are defined on the positive half-line and we deal with positive semidefinite matrices. In the matrix setting, the interval of definition of a function  may be quite important; for instance the class of operator monotone functions on the whole real line reduces to affine functions.

\section{The sums property for matrices}

 For  concave functions, the inequality \eqref{eqext1} is reversed. Here is the matrix version.  An isometry 
$U\in\bM_{2n,n}$ means a $2n\times n$ matrix such that $U^*U=I$, the identity of $\bM_n$.

\vskip 5pt
\begin{theorem}\label{th-HH1}  Let $A, B\in\bM_{n}^+$, let $0<x<1$, and let $f(t)$ be a monotone concave function  on $[0,\infty)$ with $f(0)\ge 0$. Then, for some isometry matrices $U,V\in\bM_{2n,n}$,
$$
f(A)\oplus f(B) \le Uf((1-x)A+xB)U^* + Vf(xA+(1-x)B)V^*.
$$
\end{theorem}

\vskip 5pt
\begin{proof} Let $x=\sin^2 \theta$, $1-x=\cos^2\theta$, consider the unitary Hermitian matrix
$$
R:=\begin{bmatrix}
\sqrt{1-x} I& \sqrt{x}I\\  \sqrt{x}I& -\sqrt{1-x} I
\end{bmatrix}
$$
and note the unitary congruence
\begin{equation}\label{eq-cong}
R
\begin{bmatrix}
A & 0 \\ 0 & B
\end{bmatrix}R=
\begin{bmatrix} (1-x)A+xB
 & \star \\ \star &  xA+(1-x)B
\end{bmatrix}
\end{equation}
where the stars hold for unspecified entries. 

Now, recall the decomposition \cite[Lemma 3.4]{BL1} : Given any positive semidefinite matrix $\begin{bmatrix}
C & X \\ X^* & D
\end{bmatrix}$ partitioned in four blocks in $\bM_n$, we have
$$
\begin{bmatrix}
C & X \\ X^* & D
\end{bmatrix}
=
U_0\begin{bmatrix}
C & 0 \\ 0 & 0
\end{bmatrix}
U_0^*+
V_0
\begin{bmatrix}
0 & 0 \\ 0 & D
\end{bmatrix}
V_0^*
$$
for some unitary matrices $U_0,V_0\in\bM_{2n}$. Applying this to \eqref{eq-cong}, we obtain
\begin{equation}\label{eq-dec} 
\begin{bmatrix}
A & 0 \\ 0 & B\end{bmatrix} = U_1\begin{bmatrix}
(1-x)A+xB & 0 \\ 0 & 0
\end{bmatrix}
U_1^*+
V_1
\begin{bmatrix}
0 & 0 \\ 0 & xA+(1-x)B
\end{bmatrix}
V_1^*
\end{equation}
for two unitary matrices $U_1,V_1\in\bM_{2n}$.

Next, recall the subadditivity inequality \cite[Theorem 3.1]{BL1} : Given any pair of positive semidefinite matrices $S,T\in\bM_d$, we have
$$
f(S+T) \le U_2f(S)U_2^* + V_2f(T)V_2^*
$$
for two unitary matrices $U_2,V_2\in\bM_{d}$. Applying this to \eqref{eq-dec} yields
$$
\begin{bmatrix}
f(A) & 0 \\ 0 & f(B)\end{bmatrix} \le U\begin{bmatrix}
f((1-x)A+xB) & 0 \\ 0 & f(0)I
\end{bmatrix}
U^*+
V
\begin{bmatrix}
f(0)I & 0 \\ 0 & f(xA+(1-x)B)
\end{bmatrix}
V^*
$$
for two unitary matrices $U,V\in\bM_{2n}$. This proves the theorem when $f(0)=0$.

 To derive the general case,  we may assume that $f(t)$ is continuous (a concave function on $[0,\infty)$ might be discontinuous at $0$). Indeed, it suffices to consider the values of $f(t)$  on a finite set, the union of the spectra of the four matrices $A$, $B$, $(1-x)A+xB$ and $xA+(1-x)B$. Hence we may replace $f(t)$ by a piecewise affine monotone concave function $h(t)$ with $h(0)\ge 0$. Now, since $h(t)$ is continuous, a limit argument allows us to suppose that $A$ and $B$ are invertible. Therefore, letting $\lambda_n^{\downarrow}(Z)$ denote the smallest eigenvalue of $Z\in\bM_n^{s.a}$,
$$
r:=\min\{\lambda_n^{\downarrow}(A), \lambda_n^{\downarrow}(B)\}>0.
$$
We may then replace $h(t)$ by $h_r(t)$ defined as $h_r(t):=h(t)$ for $t\ge r$, $h_r(0):=0$ and $h_r(s):=h(r)\frac{s}{r}$ for $0\le s\le r$. The function $h_r(t)$ is monotone concave on $[0,\infty)$ and vanishes at $0$, thus the case $f(0)=0$ entails the general case.
\end{proof}

\vskip 5pt
Let  $\lambda_j^{\downarrow}(Z)$, $j=1,2,\ldots n$, denote the  eigenvalues of $Z\in\bM_n^{s.a}$ arranged in the nonincreasing order.

\vskip 5pt
\begin{cor}\label{coreigen1}  Let $A, B\in\bM_{n}^+$, let $0<x<1$, and let $f(t)$ be a nonnegative concave function  on $[0,\infty)$. Then, for  $j=0,1,\ldots, n-1$,
$$
\lambda_{1+2j}^{\downarrow}\left(f(A\oplus B)\right) \le \lambda_{1+j}^{\downarrow}\left(f(xA+(1-x)B) \right) + \lambda_{1+j}^{\downarrow}\left(f((1-x)A+xB)\right)
$$
and
$$
\lambda_{1+j}^{\downarrow}\left\{\left(f(xA+(1-x)B) \right) + \left(f((1-x)A+xB)\right)\right\}
 \le 2\lambda_{1+j}^{\downarrow}\left(f\left(\frac{A+B}{2}\right) \right).
$$
\end{cor}

\vskip 5pt
\begin{proof} The first inequality is a straighforward consequence of Theorem \ref{th-HH1} combined with the  inequalities of Weyl \cite[p.\ 62]{Bh} :
For all $S,T\in\bM_d^{s.a}$ and $j,k\in\{0,\ldots,d-1\}$ such that $j+k+1\le d$,
$$
\lambda_{1+j+k}^{\downarrow}(S+T) \le 
\lambda_{1+j}^{\downarrow}(S) + 
\lambda_{1+k}^{\downarrow}(T).
$$
The second inequality is not new; it follows from Theorem \ref{th-convex}.
\end{proof}

\vskip 5pt
\begin{cor}\label{corSchatt1}  Let $A, B\in\bM_{n}^+$, let $0<x<1$, and let $f(t)$ be nonnegative concave function  on $[0,\infty)$. Then, for all $p\ge 1$,
$$
\left(\|f(A)\|_p^p + \|f(B)\|_p^p\right)^{1/p} \le \|f(xA+(1-x)B))\|_p+ \|f((1-x)A+xB)\|_p.
$$
\end{cor}

\vskip 5pt
\begin{proof} From Theorem \ref{th-HH1} we have 
$$
\|f(A)\oplus f(B)\|_p \le \|Uf((1-x)A+xB)U^* + Vf(xA+(1-x)B)V^*\|_p.
$$
The triangle inequality for $\|\cdot\|_p$ completes the proof.
\end{proof}

\vskip 5pt
Corollary \ref{corSchatt1} with $f(t)=t^q$ reads as the following trace inequality.

\vskip 5pt
\begin{cor}\label{cortrace1}  Let $A, B\in\bM_{n}^+$ and $0<x<1$. Then, for all $p\ge 1\ge q\ge 0$,
$$
\left\{{\mathrm{Tr\,}}A^{pq} + \mathrm{Tr\,}B^{pq}\right\}^{1/p} \le \left\{{\mathrm{Tr\,}}(xA+(1-x)B)^{pq}\right\}^{1/p}+ \left\{{\mathrm{Tr\,}}((1-x)A+xB))^{pq}\right\}^{1/p}.
$$
\end{cor}

\vskip 5pt
Choosing in Corollary \ref{cortrace1} $q=1$ and $x=1/2$ yields McCarthy's inequality,
$$
{\mathrm{Tr\,}}A^{p} +{\mathrm{Tr\,}}B^{p} \le {\mathrm{Tr\,}}(A+B)^{p}. $$
This shows that Theorem \ref{th-HH1} is already significant with $f(t)=t$.
Our next corollary, for convex functions, is equivalent to Theorem \ref{th-HH1}.

\vskip 5pt
\begin{cor}\label{cor-HH1}  Let $A, B\in\bM_{n}^+$, let $0<x<1$, and let $g(t)$ be a monotone convex function  on $[0,\infty)$ with $g(0)\le 0$. Then, for some isometry matrices $U,V\in\bM_{2n,n}$,
$$
g(A)\oplus g(B) \ge Ug((1-x)A+xB)U^* + Vg(xA+(1-x)B)V^*.
$$
\end{cor}

\vskip 5pt
Since the Schatten $q$-quasinorms $\|\cdot\|_q$, $0<q<1$, are superadditive functionals on $\bM_n^+$ (see \cite[Proposition 3.7]{BH1} for a stronger statement), Corollary \ref{cor-HH1} yields the next one.

\vskip 5pt
\begin{cor}\label{corSchatt2} Let $A, B\in\bM_{n}^+$, let $0<x<1$, and let $g(t)$ be a nonnegative convex function  on $[0,\infty)$ with $g(0)\le 0$. Then, for all $0<q<1$,  
$$
\left(\|g(A)\|_q^q+ \|g(B)\|_q^q\right)^{1/q} \ge \|g(xA+(1-x)B))\|_q+ \|g((1-x)A+xB)\|_q.
$$
\end{cor}

\vskip 5pt
From the first inequality of Corollary \ref{coreigen1} we also get the following statement.

\vskip 5pt
\begin{cor}\label{coreigen2}  Let $A, B\in\bM_{n}^+$ and let $f(t)$ be a nonnegative concave function  on $[0,\infty)$. Then, for $j=0,1,\ldots, n-1$,
$$
\lambda_{1+2j}^{\downarrow}\left(f(A\oplus B)\right) \le 2\int_{0}^1
\lambda_{1+j}^{\downarrow}\left(f(xA+(1-x)B) \right) \,{\mathrm{d}}x.
$$
\end{cor}

\vskip 5pt
Up to now we have dealt with  convex combinations $(1-x)A +xB$ with scalar weights. It is natural to search for extensions with matricial weights ($C^*$-convex combinations). We may
generalize Theorem \ref{th-HH1} with commuting normal weights.

\vskip 5pt
\begin{theorem}\label{th-HH2}  Let $A, B\in\bM_{n}^+$ and let $f(t)$ be a monotone concave function  on $[0,\infty)$ with $f(0)\ge 0$. If $X,Y\in\bM_n$ are normal and satisfy  $XY=YX$ and $X^*X+Y^*Y=I$, then, for some isometry  matrices $U,V\in\bM_{2n,n}$,
$$
f(A)\oplus f(B) \le Uf(X^*AX+Y^*BY)U^* + Vf(Y^*AY+X^*BX)V^*.
$$
\end{theorem}

\vskip 5pt
\begin{proof} The proof is quite similar to that  of Theorem \ref{th-HH1} except that we first observe that the $2n\times 2n$ matrix
$$
H:=\begin{bmatrix}
X & Y \\ Y & -X
\end{bmatrix}
$$
is unitary. Indeed for two normal operators, $XY=YX$ ensures $X^*Y=YX^*$ and a direct computation shows that $H^*H$ is the identity in $\bM_{2n}$.
We then use the unitary congruence
\begin{equation}\label{eq-cong2}
H^*
\begin{bmatrix}
A & 0 \\ 0 & B
\end{bmatrix}H=
\begin{bmatrix} X^*AX+Y^*BY
 & \star \\ \star &  Y^*AY+X^*BX
\end{bmatrix}
\end{equation}
where the stars hold for unspecified entries. 
\end{proof}

\vskip 5pt
Hence, in the first inequality of Corolloray \ref{coreigen1} and in the series of Corollaries \ref{corSchatt1}-\ref{corSchatt2}, we can replace the scalar convex combinations $(1-x)A+B$ and $xA+(1-x)B$ by $C^*$-convex combinations $X^*AX+Y^*BY$ and $Y^*AY + X^*BX$ with commuting normal  weights. Here we explicitly state the generalization of Corollary \ref{cor-HH1}.

\vskip 5pt
\begin{cor}\label{cor-HH2}  Let $A, B\in\bM_{n}^+$ and let $g(t)$ be a monotone convex function  on $[0,\infty)$ with $g(0)\le 0$.  If $X,Y\in\bM_n$ are normal and satisfy  $XY=YX$ and $X^*X+Y^*Y=I$, then, for some isometry  matrices $U,V\in\bM_{2n,n}$,
$$
g(A)\oplus g(B) \ge Ug(X^*AX+Y^*BY)U^* + Vg(Y^*AY+X^*BX)V^*.
$$
\end{cor}

\section{Majorization}

The  results of Section 2 require two essential assumptions : to deal with positive matrices and with subadditive (concave) or superadditive (convex) functions. Thanks to these assumptions, we have obtained operator inequalities for the usual order in the positive cone.

The results of this section will consider Hermitian matrices and general convex or concave functions. We will  obtain majorization relations. We also consider $C^*$-convex combinations more general than those with commuting normal weights.  

We  recall the notion of majorization. Let $A,B\in\bM_n^{s.a}$.
We say that $A$ is weakly majorized by $B$ and we write $A\prec_{w}B$, if 
$$
\sum_{j=1}^k\lambda_j^{\downarrow}(A) \le  \sum_{j=1}^k\lambda_j^{\downarrow}(B)
$$
for all $k=1,2,\ldots, n$. If
 furthemore  the equality holds for $k=n$, that  is $A$ and $B$ have the same trace, then we say that $A$  is majorized by $B$, written $A\prec B$. See  \cite[Chapter 2]{Bh} and \cite{Hi} for a background on majorization. One easily checks that $A\prec_wB$ is equivalent to $A+C\prec B$ for some $C\in\bM_n^+$. We  need two fundamental  principles:
\begin{itemize}
\item[(1)]
  $A\prec B \Rightarrow g(A) \prec_w g(B)$ for all convex functions $g(t)$. 

\item[(2)] 
$A\prec_w B \iff {\mathrm{Tr\,}} f(A) \le {\mathrm{Tr\,}} f(B) $ for all nondecreasing convex functions $f(t)$. Equivalently $A\prec B \iff {\mathrm{Tr\,}} g(A) \le {\mathrm{Tr\,}} g(B) $ for all convex functions $g(t)$. 
\end{itemize}

\vskip 5pt
\begin{lemma}\label{lemma-trace}  Let $A, B\in\bM_{n}^{s.a}$ and let $g(t)$ be a convex function defined on an interval containing the spectra of $A$ and $B$. If $X,Y\in\bM_n$ satisfy $X^*X+Y^*Y=XX^*+YY^*=I$, then,
$$
{\mathrm{Tr\,}}\left\{g(X^*AX+Y^*BY)+g(Y^*AY+X^*BX)\right\}\le {\mathrm{Tr\,}} \left\{g(A)+g(B)\right\}.
$$
\end{lemma}

\vskip 5pt
\begin{proof}
By the famous Hansen-Pedersen trace inequality \cite{HP}, see also \cite[Corollary 2.4]{BL1} for a generalization,
\begin{align*}
{\mathrm{Tr\,}}&\left\{g(X^*AX+Y^*BY)+g(Y^*AY+X^*BX)\right\} \\
&\le {\mathrm{Tr\,}}\left\{X^*g(A)X+Y^*g(B)Y)+Y^*g(A)Y+X^*g(B)X\right\} \\
&= {\mathrm{Tr\,}}\left\{(g(A) + g(B))(XX^* +YY^*)\right\}= {\mathrm{Tr\,}} \left\{g(A)+g(B)\right\}
\end{align*}
where the first equality follows from the cyclicity of the trace.
\end{proof}

\vskip 5pt
\begin{theorem}\label{th-orbit}  Let $A, B\in\bM_{n}^{s.a}$ and $X,Y\in\bM_n$. If  $X^*X+Y^*Y=XX^*+YY^*=I$, then, for some  unitary matrices $\{U_k\}_{k=1}^{2n}$ in $\bM_{2n}$,
$$
(X^*AX+Y^*BY)\oplus (Y^*AY+X^*BX) =\frac{1}{2n}\sum_{k=1}^{2n} U_k (A\oplus B)U_k^*.
$$
\end{theorem}

\vskip 5pt
\begin{proof} From Lemma \ref{lemma-trace}, we have the trace inequality
$$
{\mathrm{Tr\,}}g\left((X^*AX+Y^*BY)\oplus (Y^*AY+X^*BX) \right)\le  {\mathrm{Tr\,}}g(A\oplus B) 
$$
for all convex functions defined on $(-\infty,\infty)$. By a basic principle of majorization, this is equivalent to
\begin{equation}\label{eqmaj}
(X^*AX+Y^*BY)\oplus (Y^*AY+X^*BX) \prec A\oplus B.
\end{equation}
By \cite[Proposition 2.6]{BL2}, the majorization in $\bM_d^{s.a}$, $S\prec T$, ensures that (and thus is equivalent to)
$$
S\le \frac{1}{d} \sum_{j=1}^d V_jTV_j^*
$$
for $d$ unitary matrices $V_j\in\bM_d$. Applying this to \eqref{eqmaj} completes the proof.
 \end{proof}

\vskip 5pt
\begin{remark} We can prove the majorization \eqref{eqmaj} in a different way by oberving that our assumption on $X$ and $Y$ ensures that the map $\Phi$, defined on $\bM_{d}$, (here $d=2n$),
$$
\Phi \left(\begin{bmatrix} A&C \\ D&B\end{bmatrix}\right):=\begin{bmatrix}X^*AX+Y^*BY&0 \\ 0&Y^*AY+X^*BX\end{bmatrix},
$$
is a  positive linear map unital and trace preserving. Such maps are also called doubly stochastic. It is a classical result (see Ando's survey \cite[Section 7]{An1} and references therein) that we have
\begin{equation}\label{eqchan}
\Phi(Z) \prec Z
\end{equation}
for every doubly stochastic map on $\bM_d$ and Hermitian $Z$. Here, our map is even completely positive (it is a so called doubly stochastic map, or a unital quantum channel). It is well known in the litterature (\cite[Theorem 7.1]{An1}) that we have then 
$$
\Phi(Z)=\sum_{j=1}^{m} t_j U_j^* ZU_j
$$
for some convex combination $0<t_j\le 1$, $\sum_{j=1}^{m} t_j=1$, and some unitary matrices $U_j$ (these scalars $t_i$ and matrices $U_i$ depend on $Z$). In 2003, Zhan \cite{Zh}  noted that we can take $m=d$. That we can actually take an average,
$$
\Phi(Z)=\frac{1}{d}\sum_{j=1}^{d}  U_j^* ZU_j,
$$
follows from the quite  recent observation \cite[Proposition 2.6]{BL2} mentioned in the proof of Theorem \ref{th-orbit}. For unital quantum channels, one has the Choi-Kraus decomposition (see \cite[Chapter 3]{Bh2})
$$
\Phi(A) =\sum_{i=1}^{d^2} K_iAK_i^*
$$
for some weights $K_i\in\bM_d$ such that $\sum_{i=1}^{d^2} K_iK_i^*=\sum_{i=1}^{d^2} K_i^*K_i=I$. So, the proof of
Lemma 3.1  shows that, for unital quantum channels, one may derive the fundamental majorization \eqref{eqchan} from two results for convex functions: the basic principle of majorisation and Hansen-Pedersen's trace inequality. 
\end{remark}

\vskip 5pt
Theorem \ref{th-orbit} combined with a variation of Theorem \ref{th-convex}, \cite[Corolloray 2.4]{BL1}, provide a number of interesting operator inequalities. The next corollary can be regarded as another matrix version of the scalar inequality \eqref{eqext1}. We will give a proof independent of \cite[Corollary 2.4]{BL1}.

\vskip 5pt
\begin{cor}\label{cor-ext2}  Let $A, B\in\bM_n^{s.a}$ and let $f(t)$ be a convex function defined on an interval containing the spectra of $A$ and $B$.
 If $X,Y\in\bM_n$ satisfy $X^*X+Y^*Y=XX^*+YY^*=I$, then, for some  unitary matrices $\{U_k\}_{k=1}^{2n}$ in $\bM_{2n}$,
$$
f\left(X^*AX+Y^*BY\right)\oplus f\left(Y^*AY+X^*BX \right) \le\frac{1}{2n}\sum_{k=1}^{2n} U_k f\left(A\oplus B\right)U_k^*.
$$
\end{cor}

\vskip 5pt
In particular,  for the absolute value, we note that
$$
\left|X^*AX+Y^*BY\right|\oplus \left|Y^*AY+X^*BX \right| \le\frac{1}{2n}\sum_{k=1}^{2n} U_k \left|A\oplus B\right|U_k^*.
$$

\vskip 5pt
\begin{proof} The majorization \eqref{eqmaj} and the basic principle of majorizations show
that
$$
f\left(X^*AX+Y^*BY\right)\oplus f\left(Y^*AY+X^*BX \right) \prec_w f(A\oplus B)
$$
for all convex functions defined on an interval containing the spectra of $A$ and $B$.
Thus
$$
f\left(X^*AX+Y^*BY\right)\oplus f\left(Y^*AY+X^*BX \right) +C \prec f(A\oplus B)
$$
for some positive semidefinite matrix $C\in\bM_{2n}^+$. Hence, as in the previous proof,
$$f\left(X^*AX+Y^*BY\right)\oplus f\left(Y^*AY+X^*BX \right) +C \le \frac{1}{2n}\sum_{k=1}^{2n} U_k f\left(A\oplus B\right)U_k^*
$$
for some family of unitary matrices $U_k\in\bM_{2n}$.
\end{proof}

\vskip 5pt
\begin{cor}\label{cordet}  Let $A, B\in\bM_n^+$.  If $X,Y\in\bM_n$ satisfy $X^*X+Y^*Y=XX^*+YY^*=I$, then,
$$
\det(X^*AX+Y^*BY)\det (Y^*AY+X^*BX) \ge \det A\det B.
$$
\end{cor}

\vskip 5pt
\begin{proof} Since the classical Minkowski functional $Z\mapsto \det^{1/2n} Z$ is concave on $\bM_{2n}^+$, the result is an immediate consequence of  Theorem \ref{th-orbit}.
 \end{proof}

The case $A=B$ and $X,Y$ are two orthogonal projections reads as the classical Fisher's inequality.

\vskip 5pt
Corollary \ref{cor-ext2} yields inequalities for symmetric norms and antinorms on $\bM_{2n}^+$.
Symmetric norms, $\|\cdot\|$, also called unitarily invariant norms, are classical objects in matrix analysis. We refer to \cite{Bh}, \cite{Hi}, \cite[Chapter 6]{HiP} and, in the setting of compact operators, \cite{Si}. The most famous examples are the Schatten $p$-norms, $1\le p\le \infty$, and the Ky Fan $k$-norms.

Symmetric anti-norms $\|\cdot\|_!$ are the concave counterpart of symmetric norms. Famous examples are the Schatten $q$-quasi norms, $0<q<1$ and the Minkowski functional considered in the proof of Corollary \ref{cordet}. We refer to \cite{BH1} and \cite[Section 4]{BH2} for much more examples.

\vskip 5pt
\begin{cor}\label{corfinal}  Let $f(t)$ and $g(t)$ be  two nonnegative functions defined on  $[a,b]$ and let $A,B\in\bM_n$ be Hermitian with spectra in $[a,b]$.  If $X,Y\in\bM_n$ satisfy $X^*X+Y^*Y=XX^*+YY^*=I$, then :
\begin{itemize}
\item[(i)] If $f(t)$ is concave, then, for all symmetric antinorms $\|\cdot\|_!$,
$$
 \left\| f\left(X^*AX+Y^*BY\right)\oplus f\left(Y^*AY+X^*BX \right) \right\|_!\ge \left\|  f\left(A\oplus B\right)\right\|_!.
$$
\item[(ii)] If $g(t)$ is convex, then, for all symmetric norms $\|\cdot\|$,
$$
\left\| g\left(X^*AX+Y^*BY\right)\oplus g\left(Y^*AY+X^*BX \right) \right\| 
\le\left\| g\left(A\oplus B\right) \right\| .
$$
\end{itemize} 
\end{cor}

\begin{proof} Since symmetric antinorms are unitarily invariant and superadditive, the first assertion follows from the version of Corollary \ref{cor-ext2} for concave version. The second assertion is an immediate consequence of  Corollary \ref{cor-ext2}. \end{proof}

We close this section by mentioning Moslehian's weak majorization which  provides a matrix version of the first inequality of Proposition \ref{propHH1}. We may restate \cite[Corollary 3.4]{Mo}
as inequalities
  for symmetric and antisymmetric norms.

\vskip 5pt
\begin{prop}\label{prop-HH-mox}  Let $f(t)$ and $g(t)$ be  two nonnegative functions defined on  $[a,b]$ and let $A,B\in\bM_n$ be Hermitian with spectra in $[a,b]$. 
\begin{itemize}
\item[(i)] If $f(t)$ is concave, then, for all symmetric antinorms $\|\cdot\|_!$,
$$
  \left\|f\left(\frac{A+B}{2}\right)\right\|_!
\ge
 \left\|\int_0^1 f((1-x)A+xB)\, {\mathrm{d}}x\right\|_!.
$$
\item[(ii)] If $g(t)$ is convex, then, for all symmetric norms $\|\cdot\|$,
$$
\left\|g\left(\frac{A+B}{2}\right)\right\|
\le\left\|\int_0^1 g((1-x)A+xB)\, {\mathrm{d}}x\right\| .
$$
\end{itemize}
\end{prop}

\vskip 15pt
\noindent
Jean-Christophe Bourin

\noindent
Laboratoire de math\'ematiques, 

\noindent
Universit\'e de Bourgogne Franche-Comt\'e, 

\noindent
25 000 Besan\c{c}on, France.

\noindent
Email: jcbourin@univ-fcomte.fr

  \vskip 10pt
\noindent
Eun-Young Lee

\noindent Department of mathematics, KNU-Center for Nonlinear
Dynamics,

\noindent
Kyungpook National University,

\noindent
 Daegu 702-701, Korea.

\noindent Email: eylee89@knu.ac.kr


\begin{thebibliography}{99}
{\small

\bibitem{An1} T.\ Ando, 
 Majorization, doubly stochastic matrices, and comparison of eigenvalues, {\it Linear Algebra Appl.}\ 118 (1989), 163--248.

 \bibitem{Bh} R.\ Bhatia, Matrix Analysis, Gradutate Texts in Mathematics, Springer, New-York, 1996.

\bibitem{Bh2} R.\ Bhatia,  Positive Definite Matrices, Princeton University press, Princeton 2007.

\bibitem{BH1} J.-C.\ Bourin and F. Hiai,
Norm and anti-norm inequalities for positive semi-definite matrices, {\it Internat.\ J.\ Math.}\ 22 (2011), no.\ 8, 1121--1138. 

\bibitem{BH2} J.-C.\ Bourin and F. Hiai,  Jensen and Minkowski inequalities for operator means and anti-norms, {\it Linear Algebra Appl.}\ 456 (2014), 22--53.

\bibitem{BL1} J.-C.\ Bourin and E.-Y.\ Lee,  Unitary orbits of Hermitian operators with convex or concave functions,  {\it Bull.\ Lond.\ Math.\ Soc.}\ 44 (2012), no.\ 6, 1085--1102. 

\bibitem{BL2} J.-C.\ Bourin and E.-Y.\ Lee, A Pythagorean theorem for partitioned matrices,
{\it Proc.\ Amer.\ Math.\ Soc.}, in press.

\bibitem{HP}  F.\ Hansen and G.\ K.\ Pedersen, Jensen's operator inequality, {\it Bull.\ London Math.\ Soc.}\ 35 (2003), no.\ 4, 553--564. 

\bibitem{Hi} F.\ Hiai, Matrix analysis: matrix monotone functions, matrix means, and majorization. {\it Interdiscip.\ Inform.\ Sci}.\ 16 (2010), no.\ 2, 139--248.

\bibitem{HiP} F.\ Hiai, D.\ Petz, Introduction to matrix analysis and applications. Universitext. Springer, Cham; Hindustan Book Agency, New Delhi, 2014.

\bibitem{Mo} M.\ S.\ Moslehian, 
Matrix Hermite-Hadamard type inequalities, {\it Houston J.\ Math.}\ 39 (2013), no.\ 1, 177--189.

\bibitem{Ni}
C.\ P.\ Niculescu; L.-E.\ Persson,  Old and new on the Hermite-Hadamard inequality, {\it Real Anal.\ Exchange}\ 29 (2003/04), no.\ 2, 663--685.

\bibitem{Si} B.\ Simon, Trace ideals and their applications. Second edition. Mathematical Surveys and Monographs, 120. American Mathematical Society, Providence, RI, 2005. 
}

\bibitem{Zh} X.\ Zhan, The sharp Rado theorem for majorizations, {\it Amer.\ Math.\ Monthly} 110 (2003),
152--153.

\end{thebibliography}
\end{document}